\documentclass[12pt,english]{article}

\usepackage{amsmath,amsthm,amssymb,calc,eucal, latexsym}
\usepackage[dvips]{graphics, graphicx}

\usepackage[T2A]{fontenc}
\usepackage[english]{babel}
\usepackage{latexsym}
\renewcommand{\ge}{\geqslant}
\renewcommand{\le}{\leqslant}

\newcommand{\floor}[1]{\left[ {#1} \right]}
\newcommand{\ost}[1]{\left\{ {#1} \right\}}
\newcommand{\abs}[1]{\left| {#1} \right|}
\newcommand{\skob}[1]{\left( {#1} \right)}
\def\eps{\varepsilon}
\newtheorem{Th}{Theorem}
\newtheorem{Lemma}[Th]{Lemma}
\newtheorem*{Hyp}{Conjecture}
\newtheorem{note}{Remark}

\def\F{\Bbb F}
\def\C{\Bbb C}
\def\E{\mathsf {E}}

\author{I.~D.~Shkredov, A.~S.~Volostnov}
\title{ Sums of multiplicative characters with additive convolutions
\footnote{
This work is supported by grant Russian Scientific Foundation RSF 14--11--00433.}
}
\date{}
\begin{document}
\maketitle

\bigskip

\begin{center}
    Annotation.
\end{center}

{\it \small
    In the paper we obtain new estimates for binary and ternary sums of multiplicative characters with additive convolutions of characteristic functions of sets, having small additive doubling. In particular, we improve a result of M.--C. Chang. The proof uses Croot--Sisask almost periodicity lemma.
}

\section{Introduction}

Let $p$ be a prime number, $\F_p$ be the prime field and $\chi$ be a nontrivial multiplicative character modulo $p$.
In the paper we consider a problem of obtaining good upper bounds for the exponential sum 
\begin{equation}\label{f:def_sum}
    \sum_{a\in A,\, b\in B}\chi(a+b) \,,
\end{equation}
where $A,B$ are arbitrary subsets of the field  $\F_p$.
Exponential sums of such a type were studied by various authors, see e.g. \cite{Chang}, \cite{DE1}, \cite{Kar}--\cite{Kar2}.
There is a well--known hypothesis on sums (\ref{f:def_sum}) which is called the graph Paley conjecture, see the history of the question in \cite{Chang} or  \cite{Shkr_res}, for example.
\begin{Hyp}[Paley graph]
 Let $\delta>0$ be a real number, $A,B\subset\mathbb{F}_p$ be arbitrary sets with  $\abs{A}>p^{\delta}$ and $\abs{B}>p^\delta$. Then there exists a number $\tau=\tau(\delta)$ such that for any sufficiently 
 large prime number $p$ and all nontrivial characters $\chi$ 
 the following holds 
\begin{equation}
    \abs{\sum_{a\in A,\, b\in B}\chi(a+b)}<p^{-\tau}\abs{A}\abs{B} \,.
\label{Paley}\end{equation}
\end{Hyp}

Let us say a few words about the name of the hypothesis.
The \emph{Paley graph} is the graph $G(V,E)$ with the vertex set  $V=\mathbb{F}_p$ and the set of edges $E$ such that $(a,b)\in E$ iff $a-b$ is a quadratic residue. To make the graph non--oriented we assume that  $p\equiv1\pmod 4$. 
Under these conditions if one put $B=-A$  in (\ref{Paley}) 
and  take $\chi$ equals the Legendre symbol
then an interesting statement would follow: the size of the maximal clique in the Paley graph (as well as its independent number) grows 
slowly than   $p^\delta$ for any positive $\delta$.

Unfortunately, at the moment we know few facts about the hypothesis.
 An affirmative answer 
 was obtained 
 just in the situation  $\abs{A}>p^{\frac12+\delta}$, $\abs{B}>p^{\delta}$, see \cite{Kar}---\cite{Kar2}. Even in the case $\abs{A}\sim\abs{B}\sim p^{\frac12}$ inequality~(\ref{Paley})
  is unknown, see \cite{Kar2}. 
  However, nontrivial bounds of sum (\ref{f:def_sum}) can be obtained for structural sets $A$ and $B$ with  weaker restrictions for the sizes of the sets, see~\cite{Chang}, \cite{FI}, \cite{Kar}.
Thus, in paper~\cite{Chang} Mei--Chu Chang proved such an estimate provided one of the sets 
 $A$ or $B$ has small sumset.
Recall that the {\it sumset} of two sets $X, Y \subseteq \mathbb{F}_p$ is the set
$$X+Y = \ost{x+y\,:\,x\in X, y\in Y} \,.$$
\begin{Th}[Chang]
Let $A,\,B\subset\mathbb{F}_p$ be arbitrary sets, $\chi$  be a nontrivial multiplicative character modulo $p$
and $K,\delta$ be positive numbers 
with 
\begin{gather*}
\abs{A}>p^{\frac49+\delta},\\
\abs{B}>p^{\frac49+\delta},\\
\abs{B+B}<K|B| \,.
\end{gather*}
Then there exists $\tau=\tau(\delta, K)>0$ such that the inequality
$$\abs{\sum_{a\in A,\, b\in B}\chi(a+b)}<p^{-\tau}\abs{A}\abs{B}$$
holds for all $p>p(\delta, K)$.
\label{t:Chang}
\end{Th}

In our paper we refine Chang's assumption  $\abs{A}>p^{\frac49+\delta},\abs{B}>p^{\frac49+\delta}$ and prove the following theorem. 
\begin{Th}[Main result]\label{main_theorem}
Let $A,\,B\subset\mathbb{F}_p$ be sets and $K,L,\delta > 0$ be numbers with
\begin{gather}
\label{cond:first}
\abs{A}>p^{\frac{12}{31}+\delta},\\
\abs{B}>p^{\frac{12}{31}+\delta},\\
\abs{A+A}<K\abs{A},\\
\abs{A+B}<L\abs{B}.
\label{cond:last}
\end{gather}
Then for any nontrivial multiplicative character $\chi$ modulo $p$ one has 
\begin{equation}\label{f:main_theorem}
    \abs{\sum_{a\in A,\, b\in B}\chi(a+b)}\ll\sqrt{\frac{L\log 2K}{\delta\log p}} \cdot \abs A\abs B
\end{equation}
provided $p>p(\delta, K, L)$.
\end{Th}

 Of course our result is not a direct improvement of Chang's theorem because of the additional assumption $\abs{A+B}<L\abs B$.
 However it is applicable in the case $B=-A$ and hence in terms of the Paley graph our result is better.
On the other hand, the Pl\"{u}nnecke--Ruzsa triangle inequality (see Theorem \ref{t:Ruzsa_triangle} below) implies that the restriction $|A+B|\le L|B|$ gives us $|A+A| \le L^2 |A| \cdot (|B|/|A|)^2$ and hence if $A$ and $B$ have comparable sizes then it is enough  to assume condition (\ref{cond:last}) in Theorem \ref{main_theorem}.
Nevertheless the dependence on $K$ and $L$ in formula (\ref{f:main_theorem}) is asymmetric and thus the formulation of  our results in terms of these two parameters is reasonable.

Our approach uses a remarkable Croot--Sisask lemma \cite{Croot-Sisask} on almost periodicity of convolutions 
of the characteristic functions of sets.
Thanks to the result we reduce sum (\ref{f:main_theorem}) to a sum with more variables.
It seems like that it is the first application of the lemma in Analytical Number Theory.

In paper \cite{Hanson} B. Hanson obtained a bound for  so--called ternary sum.

\begin{Th}\label{t:Hanson}
    Let $A,\,B,\,C\subset\mathbb{F}_p$ be any sets, $\chi$  be a nontrivial multiplicative character modulo $p$.
    Suppose that for $\zeta > 0$ one has $|A|,|B|,|C| > \zeta \sqrt{p}$.
    Then
\begin{equation}\label{f:Hanson}
    \abs{\sum_{\substack{a\in A,\,b\in B,\,c\in C}}\chi(a+b+c)} = o_\zeta (|A||B||C|) \,.
\end{equation}
\end{Th}

Using the method of the proof of Theorem \ref{t:Chang} as well as some last results from sum---product theory \cite{Shkr2},
we obtain an upper  bound for the ternary sum in the case of sets with small additive doubling.

\begin{Th}\label{ternarnic}
Suppose that $A,\,B,\,C\subset\mathbb{F}_p$ are arbitrary sets and $K,L,\delta > 0$  are real numbers such that
\begin{gather}
\label{0}\abs{A},\abs{B},\abs{C}>p^{\frac{12}{31}+\delta},\\
\label{1}\abs{A+A}<K|A|,\\
\label{-1}\abs{B+C}<L\abs B.
\end{gather}
Then  there exists $\tau=\tau(\delta, K)=\delta^2(\log 2K)^{-3+o(1)}$ with the property 
$$\abs{\sum_{\substack{a\in A,\,b\in B,\,c\in C}}\chi(a+b+c)}<p^{-\tau}\abs A\abs B\abs C$$
for all $p>p(\delta, K, L)$. Here $\chi$  is a nontrivial multiplicative character modulo $p$.
\end{Th}

From the proof of Theorem \ref{ternarnic} it follows that a nontrivial upper bound in formula (\ref{f:Hanson}) requires the restriction
$\zeta \gg \exp(-(\log p)^{\alpha})$, where $\alpha >0$ is an absolute constant.

\section*{Definitions and notation}


Recall that the (Minkowski) {\it sumset}  of two sets $A$ and $B$ from the field $\mathbb{F}_p$ 
is the set 
$$A+B = \ost{a+b\,:\,a\in A, b\in B} \,.$$
In a similar way one can define the {\it difference}, the {\it product} and the {\it quotient set} of two sets $A$ and $B$ as
$$A-B = \ost{a-b\,:\,a\in A, b\in B};$$
$$AB = \ost{ab\,:\,a\in A, b\in B};$$
$$\frac AB = \ost{ab^{-1}\,:\,a\in A, b\in B, b\neq 0} \,.$$
Also for an arbitrary $g\in \mathbb{F}_p$ by  $g+A$ and $gA$ denote the sumset $\ost{g}+A$
and the product set
$\ost{g} \cdot A$, correspondingly.
We need the remarkable  Pl\"{u}nnecke--Ruzsa triangle inequality (see~\cite{TV}, p.79 and section 6.5 here).

\begin{Th}[Pl\"{u}nnecke--Ruzsa]
For any nonempty sets $A,B,C$ one has
$$\abs{A-C}\le\frac{\abs{A-B}\abs{B-C}}{\abs B}$$
and
$$\abs{A+C}\le\frac{\abs{A+B}\abs{C+B}}{\abs B} \,.$$
\label{t:Ruzsa_triangle}
\end{Th}

Besides, we denote
$$[a,b]=\ost{i\in\mathbb{Z}\,:\,a\le i\le b} \,.$$
Let  $A$ be an arbitrary set.
We write 
$A(x)$ for the characteristic function of $A$.
In other words 
$$A(x) = \begin{cases}
1, &\text{if  $x\in A$;}\\
0, &\text{otherwise.}
\end{cases}$$
We need in the notion of the {\it convolution} of two functions $f, g : \F_p \to \C$
$$(f\ast g)(x)=\sum_{y}f(y)g(x-y) \,.$$
$L_p$--norm of a function $f :\F_p \to \C$ is given by 
$$\|f\|_{L_p}=\skob{\sum_{x} |f(x)|^p}^{\frac 1p}  \,.$$
Also we will use 
the {\it multiplicative energy} of a set $A$, see \cite{TV}
$$\E(A) = \E^\times (A) = \abs{\ost{(a_1,a_2,a_3,a_4)\in A^4\,:\,a_1a_2=a_3a_4}} $$
and the {\it additive energy} of $A$ \cite{TV}
$$
    \E^{+} (A) = \abs{\ost{(a_1,a_2,a_3,a_4)\in A^4\,:\,a_1+a_2=a_3+ a_4}} \,.
$$

A {\it generalized arithmetic progression of dimension $d$} is a set $P\subset\mathbb{F}_p$
of the form 
\begin{equation}P=a_0+\ost{\sum_{j=1}^dx_ja_j\,:\,x_j\in\floor{0,H_j-1}} \,,\label{representation}\end{equation}
where  $a_0,a_1,\ldots,a_d$ are 
some elements from $\mathbb{F}_p$; $P$ is said to be {\it proper} if all of the sums in (\ref{representation}) 
are distinct (in the case $\abs P=\prod_{j=1}^d H_j$).

\begin{Th}[Freiman]
 For any set $A\subseteq \mathbb{F}_p$ such that $\abs{A+A}\le K\abs A$ there is a generalized arithmetic progression $P$ of dimension $d$ containing  $A$ such that $d\le C(K)$ and $\abs P\le e^{C(K)}\abs A$.
 Here $C(K) >0$ is a constant which depends on $K$ only but not on the set $A$.
\end{Th}

It is known that the constant $C(K)$ can be taken equal $C(K) = (\log 2K)^{3+o(1)}$, see \cite{Sanders_survey}.

\bigskip

Also let us remind that a  multiplicative character $\chi$  modulo $p$ is a homomorphism from $\mathbb{F}^{*}_p$ into 
the unit circle of the complex plane.
The character $\chi_0\equiv1$ is called trivial and the  conjugate to a character 
 $\chi(x)$ is the character  $\overline{\chi}(x)=\overline{\chi(x)}=\chi(x^{-1})$.
The order of a character $\chi$ is the least positive integer $d$ such that $\chi^d=\chi_0$. One can read about properties of multiplicative characters in \cite{Stepanov} or \cite{IK}.

We need a variant of Andr\'{e} Weil's result (see  Theorem~11.23 in~\cite{IK}).

\begin{Th}[Weil]
Let $\chi$ be a nontrivial multiplicative character modulo $p$ of order $d$. 
Suppose that a polynomial $f$ has $m$ distinct roots and there is no polynomial $g$ such that  $f=g^d$.
Then
$$\abs{\sum_{x\in\mathbb{F}_p}\chi\skob{f(x)}}\le(m-1)\sqrt p \,.$$
\end{Th}

Also we will use the H\"{o}lder inequality.

\begin{Lemma}[The H\"{o}lder inequality]
For any positive $p$ and $q$ such that  $\frac1p+\frac1q=1$ one has
$$\abs{\sum_{k=1}^{n} x_ky_k}\le\skob{\sum_{k=1}^n\abs{x_k}^p}^{\frac1p}\skob{\sum_{k=1}^n\abs{y_k}^q}^{\frac1q} \,.$$
In particular, we have the Cauchy--Schwarz inequality
$$\skob{\sum_{k=1}^{n}x_ky_k}^2\le\skob{\sum_{k=1}^nx_k^2}\skob{\sum_{k=1}^ny_k^2} \,.$$
\end{Lemma}

As we said in the introduction our proof relies on the Croot--Sisask lemma, see \cite{Croot-Sisask} and \cite{Sanders_survey}.

\begin{Lemma}[Croot--Sisask]\label{Croot-Sisask}
Let $\eps\in\skob{0,1}$, $K\ge1$, $q\ge2$ be real numbers, $A$ and $S$  be subsets of an abelian group $G$ such that $\abs{A+S}\le K\abs{A}$ and let $f\in L_q(G)$  be an arbitrary function. Then  there is  $s\in S$ and a set $T\subset S-s$, $\abs{T}\ge\abs{S}\skob{2K}^{-O(\eps^{-2}q)}$ such that for all $t\in T$ 
the following holds 
\begin{equation*}
\|(f\ast A)(x+t)-(f\ast A)(x)\|_{L_q(G)}\le\eps\abs{A}\|f\|_{L_q(G)}.
\end{equation*}
\end{Lemma}

\section*{Some preliminary lemmas}

In paper~\cite{Shkr2} the following two important results were proved.

\begin{Th}\label{energy}
Suppose that $A,B,C\subset\mathbb{F}_p$ are sets with $\abs A\abs B\abs C=O(p^2)$. Then
$$
\abs{\ost{(a_1,a_2,b_1,b_2,c_1,c_2)\in A^2\times B^2\times C^2\,:\,a_1(b_1+c_1)=a_2(b_2+c_2)}}
    \ll
$$
$$
    \ll\skob{\abs A\abs B\abs C}^{\frac32}+\abs A\abs B\abs C\max\ost{\abs A,\abs B,\abs C}.
$$
\end{Th}
\begin{Th}\label{lines}
Let $P=A\times B$  be a set of  $n$ points of $\mathbb{F}_p^2$ and $\abs A,\abs B \le p^{\frac23}$. 
Then the set $P$ has  $O(n^{\frac34}m^{\frac23}+m+n)$ incidences with any $m$ lines.
\end{Th}

The results above imply  two consequences.  

\begin{Lemma}\label{energy_esteem}
For any set $A\subset\mathbb{F}_p$ such that $\abs{A \pm A}\le K\abs A$ and $|A|^3 K = O(p^2)$ one has $\E(A)\ll K^{\frac32} \abs{A}^{\frac52}$.
\end{Lemma}
\begin{proof}
Let  $S=A+A$ (the case $A-A$ is similar).
We have 
\begin{multline*}
\E(A) = \E^\times (A) = \abs{\ost{(a_1,a_2,a_3,a_4)\in A^4\,:\,a_1a_2=a_3a_4}}=\\ =\frac1{\abs A^2}\abs{\ost{(a_1,a_2,a_3,a_4,a_2',a_4')\in A^6\,:\,a_1(a_2+a_2'-a_2')=a_3(a_4+a_4'-a_4')}}\le\\
\le\frac1{\abs A^2}\abs{\ost{(a_1,a_3,a'_2,a'_4,s_1,s_2)\in A^4\times S^2\,:\,a_1(s_1-a'_2)=a_3(s_2-a'_4)}} \,.
\end{multline*}
Using Theorem \ref{energy}, we get 
\begin{equation*}
\E(A) \ll\frac{\skob{\abs A\abs A\abs S}^{\frac32}+\abs A^2\abs S^2}{\abs A^2}\ll K^{\frac32} \abs{A}^{\frac52} 
\end{equation*}
as required. 
\end{proof}

\begin{Lemma}\label{system_solution}
Suppose that 
$A,\,B,\,C\subset\mathbb{F}_p$ are any sets and $K,L$, $L\le p^{1/16}$ are positive numbers such that
\begin{gather*}
\abs{A},\abs{B},\abs{C}<\sqrt p,\\
\abs{A+A}<K|A|,\\
\abs{B+C}<L\abs B.
\end{gather*}
Then  the system of equations 
\begin{equation}\begin{cases}
\frac{b_1+c_1}a=\frac{b_1'+c_1'}{a'}\\
\frac{b_2+c_2}a=\frac{b_2'+c_2'}{a'}
\end{cases}
\label{system}
\end{equation}
has 
\begin{equation}\label{f:bound_E^t_3}
    O(K^\frac34 L^{\frac43}\abs A^{\frac54}\abs B^{\frac{17}6}\abs C^{\frac{10}3}\log^\frac12 p
        + |A|^2 |B|^2 |C|^2)
\end{equation}
solutions in the variables  $(a,a',b_1,b_1',b_2,b_2',c_1,c_1',c_2,c_2')\in A^2\times B^4\times C^4$.
\end{Lemma}
\begin{proof}
Clearly, the number of trivial solutions $b_1=-c_1$, $b'_1=-c'_1$, $b_2=-c_2$, $b'_2=-c'_2$ and $a_1,a_2\in A$ are any numbers does not exceed
$$
   |A|^2 |B\cap(-C)|^4 \le |A|^2 |B|^2 |C|^2 
$$
and this gives us the second term in (\ref{f:bound_E^t_3}).
Below we will assume that all numerators in (\ref{system}) are nonzero.

Let $S = B+C$ and for any  $\lambda\in\mathbb{F}_p$ put 
\begin{gather*}
f(\lambda) = \abs{\ost{(b,c,s)\in B\times C\times S\,:\,\lambda = \frac{b+c}s}} \,,\\
g(\lambda) = \abs{\ost{(b,b',c,c')\in B^2\times C^2\,:\,\lambda = \frac{b+c}{b'+c'}}} \,,\\
h(\lambda) = \abs{\ost{(a,a')\in A^2\,:\,\lambda = \frac a{a'}}} \,.
\end{gather*}
Obviously, each element $s$ of the set $S$ has at most $\abs C$  representations of the form $s=b+c$, where $b\in B$ and $c\in C$ and, hence, for any  $\lambda$ one has 
\begin{equation}
g(\lambda)\le\abs C f(\lambda) \,.\label{g<f}
\end{equation}
Let $\omega^2 = |C|^{4/3} |B|^{3/2} |S|^{4/3} |A|^{-3/4} K^{3/4}$.
Consider two sets 
$$
    \Lambda_1=\ost{\lambda\in\mathbb{F}_p\,:\,f(\lambda)\le
\omega
},
\quad
    \Lambda_2=((B+C)/S) \setminus\Lambda_1 \,.
$$
Since 
$$
\omega
\abs{\Lambda_2}\le\sum_{\lambda\in\Lambda_2}f(\lambda)\le\sum_{\lambda\in\mathbb{F}_p}f(\lambda)=\abs B\abs C\abs S
$$
it follows that 
\begin{equation}\label{f:L_2}
\abs{\Lambda_2}\le |B| |C| |S| \omega^{-1}
\le p^{\frac23} \,.
\end{equation}
Indeed the last inequality is equivalent to 
$$
    |B|^{1/2} |C|^{2/3} |S|^{2/3} |A|^{3/4} K^{-3/4} \le p^{4/3} 
$$
which is true because of the conditions  $|A|,|B|,|C| < \sqrt{p}$ and $|S| \le L|B| < p^{9/16}$.

Further, the systems of the equations (\ref{system}) can be rewritten in an equivalent form, namely, 
$$\frac{a}{a'}=\frac{b_1+c_1}{b_1'+c_1'}=\frac{b_2+c_2}{b_2'+c_2'} \,.$$
Whence the number of its solutions equals 
\begin{equation}
    \sum_{\lambda\in\mathbb{F}_p}g(\lambda)^2h(\lambda)
        =
            \sum_{\lambda\in\Lambda_1}g(\lambda)^2h(\lambda)+\sum_{\lambda\in\Lambda_2}g(\lambda)^2h(\lambda) \,.
\label{answer}
\end{equation}
Foremost let us estimate the first sum in (\ref{answer})
\begin{multline}
\sum_{\lambda\in\Lambda_1}g(\lambda)^2h(\lambda)
    \le
        \sum_{\lambda\in\Lambda_1}\abs{C}^2f(\lambda)^2h(\lambda)
            \le\\ \le\sum_{\lambda\in\Lambda_1}\abs{C}^2 \omega^2 h(\lambda)
                \le \omega^2 \abs{C}^2\sum_{\lambda\in\mathbb{F}_p}h(\lambda)=
                \omega^2 |A|^2 |C|^2 \,.
\label{first_sum}
\end{multline}
Further using the Cauchy--Schwarz inequality, we get for the second sum in (\ref{answer})
\begin{equation}
\sum_{\lambda\in\Lambda_2}g(\lambda)^2h(\lambda)\le\skob{\sum_{\lambda\in\Lambda_2}g(\lambda)^4}^{\frac12}\skob{\sum_{\lambda\in\Lambda_2}h(\lambda)^2}^{\frac12}.
\label{after_kbsh}
\end{equation}
By the assumption $|A|<\sqrt{p}$ and hence $|A|^3 K = O(p^2)$.
Thus by Lemma \ref{energy_esteem}, we obtain 
$$
\sum_{\lambda\in\Lambda_2}h(\lambda)^2\le\sum_{\lambda\in\mathbb{F}_p}h(\lambda)^2
    =\abs{\ost{(a_1,a_2,a_3,a_4)\,:\,\frac{a_1}{a_2}=\frac{a_3}{a_4}}}
        =
$$
\begin{equation}
        =\E(A)\ll K^\frac32 \abs{A}^{\frac52}.
\label{h_esteem}
\end{equation}
For any $\tau \ge \omega$ consider the set 
$$W_{\tau} = \ost{\lambda\in\Lambda_2\,:\,f(\lambda)\ge\tau} \,.$$
Take the set of points $P=W_{\tau}\times B$ in $\mathbb{F}^2_p$ and the set of lines $$\mathcal{L}=\ost{sx=y+c\,:\,(s,c)\in S\times C} \,.$$
Because $\abs{W_{\tau}},\abs B \le p^{\frac23}$ it follows that the number of incidences between the points $P$ and the lines $\mathcal{L}$ can be estimated by Theorem \ref{lines} as 
\begin{equation}O\skob{(\abs{W_{\tau}}\abs B)^{\frac34}(\abs S\abs C)^{\frac23}+\abs{W_{\tau}}\abs B+\abs S\abs C}.\label{incidents}\end{equation}
Further, using a trivial bound 
$|W_\tau| \le |S| |B| |C| \tau^{-1} \le |S||B||C| \omega^{-1}$, we see that the inequality
$$(\abs{W_{\tau}}\abs B)^{\frac34}(\abs S\abs C)^{\frac23}
\gg \abs{W_{\tau}}\abs B $$
is followed  from 
\begin{equation}\label{tmp:16.05.2016_2}
    \omega^3 |S|^5 |C|^5 \gg |B|^6 \,.
\end{equation}
Let us prove that the last bound takes place.
Indeed, the number of the solutions of equation (\ref{system}) can be estimated by Theorem \ref{energy} and formulas (\ref{g<f}),  (\ref{answer}) as 
$$
    |A| \sum_{\lambda} g^2 (\lambda) \le |A| |C|^2 \sum_{\lambda} f^2 (\lambda) \ll |A| |C|^2 (|C||S||B|)^{3/2} 
$$
because of $|S| \le L |B| < p^{9/16} \le p^{2/3}$.
Hence, in the light of the 
required  
estimate (\ref{f:bound_E^t_3}), we can assume that 
$$
    |S| |C| \gg |A|^{3/2} \,.
$$
But then we have  $\omega \ge |B|^{3/4}$ and thus inequality (\ref{tmp:16.05.2016_2})
holds immediately. 

Further   if
\begin{equation}\label{tmp:16.05.2016_1}
    (\abs{W_{\tau}}\abs B)^{\frac34}(\abs S\abs C)^{\frac23}
\gg
\abs S\abs C 
\end{equation}
then the number of incidences (\ref{incidents}) equals 
$$O\skob{L^{\frac23}(\abs{W_{\tau}}\abs B)^{\frac34}(\abs B\abs C)^{\frac23}} \,.$$
Finally, in view of 
\begin{multline*}
\tau\abs{W_\tau}\le\sum_{\lambda\in W_\tau}f(\lambda)=\abs{\ost{(\lambda,b,s,c)\in W_\tau\times B\times S\times C\,:\,s\lambda=b+c}}\ll\\ \ll L^{\frac23}(\abs{W_{\tau}}\abs B)^{\frac34}(\abs B\abs C)^{\frac23} \,,
\end{multline*}
we get 
\begin{equation}
\abs{W_\tau}\ll\frac{L^{\frac83}\abs B^{\frac{17}3}\abs C^{\frac83}}{\tau^4} \,.
\label{Wtau}
\end{equation}
 But if (\ref{tmp:16.05.2016_1}) does not hold then because of, trivially, 
 $\tau \le |B||C|$ one can check bound (\ref{Wtau}) directly.
So, inequality (\ref{Wtau}) takes place.

As we noted before the maximal value of  $f(\lambda)$ is at most $\abs B\abs C<p$.
Using the fact and inequality~(\ref{Wtau}), we see that 
\begin{multline*}\sum_{\lambda\in\Lambda_2}f(\lambda)^4=\sum_{j=1}^{\lceil\log p\rceil}\sum_{\substack{\lambda\in\Lambda_2\,:\\ 2^{j-1}\le f(\lambda)<2^j}}f(\lambda)^4\le\sum_{j=1}^{\lceil\log p\rceil}2^{4j}\abs{W_{2^{j-1}}}\ll\\ \ll\sum_{j=1}^{\lceil\log p\rceil}2^{4j}\frac{L^{\frac83}\abs B^{\frac{17}3}\abs C^{\frac83}}{2^{4(j-1)}}\ll L^{\frac83}\abs B^{\frac{17}3}\abs C^{\frac83}\log p \,.
\end{multline*}
Applying simple bound (\ref{g<f}), we obtain 
\begin{equation}
\sum_{\lambda\in\Lambda_2}g(\lambda)^4\le\abs C^{4}\sum_{\lambda\in\Lambda_2}f(\lambda)^4\ll L^{\frac83}\abs B^{\frac{17}3}\abs C^{\frac{20}3}\log p  \,.
\label{g_esteem}
\end{equation}
Combining inequalities~(\ref{after_kbsh}),~(\ref{h_esteem})~and~(\ref{g_esteem}),
we get
\begin{equation}
\sum_{\lambda\in\Lambda_2}g(\lambda)^2h(\lambda)\ll K^\frac34 L^{\frac43}\abs A^{\frac54}\abs B^{\frac{17}6}\abs C^{\frac{10}3}\log^\frac12 p  \,.
\label{second_sum}
\end{equation}
Altogether from (\ref{answer}),~(\ref{first_sum}), (\ref{second_sum}) and our choice of the parameter $\omega$,
we have 
$$\sum_{\lambda\in\mathbb{F}_p}g(\lambda)^2h(\lambda)\ll K^\frac34 L^{\frac43}\abs A^{\frac54}\abs B^{\frac{17}6}\abs C^{\frac{10}3}\log^\frac12 p
    +
        \omega^2 |A|^2 |C|^2
            \ll
$$
$$
    \ll
        K^\frac34 L^{\frac43}\abs A^{\frac54}\abs B^{\frac{17}6}\abs C^{\frac{10}3}\log^\frac12 p \,.
$$
This completes the proof of the lemma.
\end{proof}

Weil's Theorem implies the following result.

\begin{Lemma}\label{davenport}
For any nontrivial character $\chi$, an arbitrary  set  $I\subset\mathbb{F}_p$ and a positive integer $r$ one has
$$\sum_{u_1,u_2\in \mathbb{F}_p}\abs{\sum_{t\in I}\chi(u_1+t)\overline{\chi}(u_2+t)}^{2r}< p^2\abs{I}^rr^{2r}+4r^2p\abs{I}^{2r} \,.$$
\end{Lemma}
\begin{proof}
We have 
\begin{multline}
\sum_{u_1,u_2}\abs{\sum_{t\in I}\chi(u_1+t)\overline{\chi}(u_2+t)}^{2r}=\\
=\sum_{u_1,u_2}\sum_{t_1,\ldots,t_{2r}\in I}\chi\skob{\frac{(u_1+t_1)\cdots(u_1+t_r)(u_2+t_{r+1})\cdots(u_2+t_{2r})}{(u_2+t_1)\cdots(u_2+t_r)(u_1+t_{r+1})\cdots(u_1+t_{2r})}}=\\
=\sum_{t_1,\ldots,t_{2r}\in I}\sum_{u_1,u_2}\chi\skob{\frac{(u_1+t_1)\cdots(u_1+t_r)(u_2+t_{r+1})\cdots(u_2+t_{2r})}{(u_2+t_1)\cdots(u_2+t_r)(u_1+t_{r+1})\cdots(u_1+t_{2r})}}=\\
= \sum_{t_1,\ldots,t_{2r}\in I}\abs{\sum_{u\in\mathbb{F}_p}\chi\skob{\frac{(u+t_1)\cdots(u+t_r)}{(u+t_{r+1})\cdots(u+t_{2r})}}}^{2}.
\end{multline}
Consider a polynomial $$f(x)=(x+t_1)\cdots(x+t_r)(x+t_{r+1})^{p-2}\cdots(x+t_{2r})^{p-2} \,.$$
Then 
$$
    \abs{\sum_{u\in\mathbb{F}_p}\chi\skob{\frac{(u+t_1)\cdots(u+t_r)}{(u+t_{r+1})\cdots(u+t_{2r})}}}
        =
            \abs{\sum_{u\in\mathbb{F}_p}\chi\skob{f(u)}} \,.
$$
The polynomial $f(x)$ has at most $2r$ distinct roots.
The order $d$ of the character $\chi$ is a divisor of $p-1$ and hence it is coprime with $p-2$.
 Thus if there exists an element $t_k$ (let us call it a <<unique>> element) among the numbers  $\ost{t_i}$ with $\forall\,j\neq k$, $t_j\neq t_k$ then the polynomial $f(x)$ satisfies all conditions of Weil's Theorem and in the case, we have
$$
    \abs{\sum_{u\in\mathbb{F}_p}\chi\skob{f(u)}}<2r\sqrt p \,.
$$
Clearly, the number of tuples with a <<unique>> element does not exceed the total number of tuples $\ost{t_i}$, i.e. $\abs{I}^{2r}$. Now let us estimate the number of tuples $\ost{t_i}$ having no a <<unique>> element.
 Then, obviously, any element of such a tuple appears in it at least twice. Hence each of these tuples contains at most $r$ different elements  and thus the number of such 
 sequences 
 can be bounded as $\abs{I}^rr^{2r}$.
For any tuple without a <<unique>> element we estimate the sum $\abs{\sum\limits_{u\in\mathbb{F}_p}\chi\skob{f(u)}}$ by $p$. Whence we obtain a final bound
$$
    \sum_{t_1,\ldots,t_{2r}\in I}\abs{\sum_{u\in\mathbb{F}_p}\chi\skob{f(u)}}^{2}
        <
            \abs{I}^{2r}(2r\sqrt p)^2+\abs{I}^rr^{2r}p^2 \,.
$$
This completes the proof. 
\end{proof}

The proofs of statements which are similar to Lemma \ref{davenport} can be found in \cite{Chang} and in book \cite{IK}, see Corollary 11.24.

\section*{The proofs of the main results}

First of all, we prove Theorem~\ref{ternarnic} and after that show how it implies Theorem~\ref{main_theorem}. 

\begin{proof}[The proof of Theorem~\ref{ternarnic}]
We will assume that $|A|, |B|, |C| < \sqrt{p}$.
Clearly, one can suppose that the inequality $L\le p^{1/16}$ takes place otherwise it is nothing to prove
(see Remark \ref{p(d,K,L)} below about the dependence of the quantity  $p(\delta, K,L)$ on $L$ or just the current proof).
According the Freiman theorem on sets with small doubling 
there is a generalized arithmetic progression $A_1=a_0+P \subseteq \mathbb{F}_p$  of the dimension $d$, where $$P=\ost{\sum_{j=1}^dx_ja_j\,:\,x_j\in\floor{0,H_j-1}}$$
such that 
$$A\subset A_1$$
$$d\le C(K)$$
$$\abs{A_1}<e^{C(K)}\abs{A} \,.$$
Put 
$$\alpha=\frac{7\delta}{18d},\quad
r=\left\lceil\frac1\alpha \right\rceil\,.$$
Take the interval $I=\floor{1,p^\alpha}$ and the generalized progression $A_0$ of the dimension $d$ defined as 
$$A_0=\ost{\sum_{j=1}^dx_ja_j\,:\,x_j\in\floor{0,p^{-2\alpha}H_j}} \,.$$
Clearly, 
\begin{equation}
\abs{A_0}\ge p^{-2d\alpha}\abs{A_1}\ge p^{-2d\alpha}\abs{A}
\label{2}
\end{equation}
and 
\begin{equation}
\abs{A_0+A_0}\le 2^d\abs{A_0} \,.\label{A0+A0}
\end{equation}
Because of $A_0I\subseteq \ost{\sum\limits_{j=1}^dx_ja_j\,:\,x_j\in\floor{0,p^{-\alpha}H_j}}$ and hence
\begin{equation}\label{f:Bergess}
    A-A_0I \subseteq \ost{\sum\limits_{j=1}^dx_ja_j\,:\,x_j\in\floor{-p^{-\alpha} H_j, H_j}}
\end{equation}
we, clearly, get 
\begin{equation}\label{A-A_0I}
    \abs{A-A_0I}\le\skob{1+p^{-\alpha}}^d\abs{A_1}\le e^{C(K)}\skob{1+p^{-\alpha}}^d\abs{A}\le e^{C(K)}2^d\abs{A} \,.
\end{equation}
Let us fix $x\in A_0, y\in I$  and estimate the sum 
\begin{multline}\abs{\sum_{\substack{a\in A,\, b\in B,\\ c\in C}}\chi(a+b+c)}\le\sum_{a\in A}\abs{\sum_{\substack{b\in B,\\ c\in C}}\chi(a+b+c)}=\\ =\sum_{a\in A-xy}\abs{\sum_{\substack{b\in B,\\ c\in C}}\chi(a+b+c+xy)}\le\sum_{a\in A-A_0I}\abs{\sum_{\substack{b\in B,\\ c\in C}}\chi(a+b+c+xy)}.\label{9}\end{multline}
The numbers $x\in A_0$, $y\in I$ can be taken in such a way that the last sum in~(\ref{9}) does not exceed the mean, whence
\begin{equation}\abs{\sum_{\substack{a\in A,\, b\in B,\\ c\in C}}\chi(a+b+c)}\le\frac1{\abs{A_0}\abs{I}}\sum_{\substack{a\in A-A_0I,\\ x\in A_0,\, y\in I}}\abs{\sum_{\substack{b\in B,\\ c\in C}}\chi(a+b+c+xy)}.
\label{rand}\end{equation}

Now having  any fixed $a\in A-A_0I$, let us estimate the sum 
$$
    \sum_{x\in A_0, y\in I}\abs{\sum_{\substack{b\in B,\\ c\in C}}\chi(a+b+c+xy)}=\sum_{x\in A_0, y\in I}\abs{\sum_{\substack{b\in B_a,\\ c\in C}}\chi(b+c+xy)} \,.
$$
Here we have denoted $B_a=a+B$. 
By the Cauchy--Schwarz inequality, we get 
\begin{multline}\skob{\sum_{x\in A_0, y\in I}\abs{\sum_{\substack{b\in B_a,\\ c\in C}}\chi(b+c+xy)}}^2\le\\ \le\skob{\sum_{x\in A_0, y\in I}1}\skob{\sum_{x\in A_0, y\in I}\abs{\sum_{\substack{b\in B_a,\\ c\in C}}\chi(b+c+xy)}^2}=\\ =\abs{A_0}\abs{I}\skob{\sum_{\substack{x\in A_0,\, y\in I,\\ b_1,\, b_2\in B_a,\\ c_1,\, c_2\in C}}\chi(b_1+c_1+xy)\overline{\chi}(b_2+c_2+xy)}.
\label{kbsh}
\end{multline}
For any pair $(u_1,u_2)\in\mathbb{F}^2_p$ put 
$$\nu(u_1,u_2)=\abs{\ost{(b_1,b_2,c_1,c_2,x)\in B_a^2\times C^2\times A_0\,:\,\frac{b_1+c_1}x=u_1
\mbox{ and } \frac{b_2+c_2}x=u_2}} \,.$$
Then for any  $x\neq 0$, we have 
\begin{multline}
\sum_{\substack{x\in A_0, y\in I,\\ b_1,b_2\in B_a,\\ c_1,c_2\in C}}\chi(b_1+c_1+xy)\overline{\chi}(b_2+c_2+xy)=\\
=\sum_{\substack{x\in A_0, y\in I,\\ b_1,b_2\in B_a,\\ c_1,c_2\in C}}\chi((b_1+c_1)x^{-1}+y)\overline{\chi}((b_2+c_2)x^{-1}+y)=\\
=\sum_{u_1,u_2\in\mathbb{F}_p^2}\nu(u_1,u_2)\sum_{y\in I}\chi(u_1+y)\overline{\chi}(u_2+y)\le\\
\le\skob{\sum_{u_1,u_2}\nu(u_1,u_2)}^{1-\frac1r} \skob{\sum_{u_1,u_2}\nu(u_1,u_2)^2}^{\frac{1}{2r}} \times\\ \times \skob{\sum_{u_1,u_2}\abs{\sum_{t\in I}\chi(u_1+t)\overline{\chi}(u_2+t)}^{2r}}^{\frac1{2r}}.
\label{main_esteem}
\end{multline}
The inequality in  (\ref{main_esteem}) follows from the H\"{o}lder inequality  and the Cauchy--Schwarz inequality.
By Lemma~\ref{davenport}
\begin{multline}\skob{\sum_{u_1,u_2}\abs{\sum_{t\in I}\chi(u_1+t)\overline{\chi}(u_2+t)}^{2r}}^{\frac1{2r}}<\skob{p^2\abs{I}^rr^{2r}+4r^2p\abs{I}^{2r}}^{\frac1{2r}}\le\\ \le r\abs{I}^{\frac12}p^{\frac1r}+(2r)^{\frac1r}p^{\frac1{2r}}\abs{I}\le2rp^{\frac1{2r}}\abs{I}\,.\label{3.22}\end{multline}
The last inequality takes place because $\abs{I}\ge p^{\frac1{r}}$ and $r\ge 2$.
Further note that 
\begin{equation}
\sum_{u_1,u_2\in \mathbb{F}_p}\nu(u_1,u_2)=\abs{B}^2\abs{C}^2\abs{A_0} 
\label{3.19}
\end{equation}
and by Lemma~\ref{system_solution}, combining with inequalities (\ref{0}), (\ref{-1}), (\ref{2}), (\ref{A0+A0})
and condition (\ref{cond:first}),  we obtain 
\begin{multline}
\sum_{u_1,u_2\in \mathbb{F}_p}\nu(u_1,u_2)^2=\\
=\abs{\ost{(x,x',b_1,b_1',b_2,b_2',c_1,c_1',c_2,c_2')\,:\,\frac{b_i+c_i}x=\frac{b'_i+c_i'}{x'}\text{ ДКЪ $i=1,2$}}}\ll\\
\ll 2^{\frac34d} L^{\frac43}\abs{A_0}^{\frac54}\abs B^{\frac{17}6}\abs C^{\frac{10}3}\log^\frac12 p
    + |A_0|^2 |B|^2 |C|^2 \ll\\
\ll \skob{\abs{A_0}\abs B^2\abs C^2}^{2}2^{\frac34d} L^{\frac43}\abs{A_0}^{-\frac34}\abs B^{-\frac{7}6}\abs C^{-\frac{2}3}\log^\frac12 p\ll\\
\ll\skob{\abs{A_0}\abs B^2\abs C^2}^{2}2^{\frac34d} L^{\frac43}p^{\frac{3d\alpha}2-\skob{\frac{12}{31}+\delta}\skob{\frac34+\frac76+\frac23}}\log^{\frac12} p=\\
=
    \skob{\abs{A_0}\abs B^2\abs C^2}^{2}2^{\frac34d} L^{\frac43}p^{\frac{3d\alpha}2-\frac{31\delta}{12}-1}\log^{\frac12} p \,.
\label{3.20}
\end{multline}
Using estimates~(\ref{kbsh})---(\ref{3.20}),
we see that 
\begin{equation*}
\skob{\sum_{x\in A_0, y\in I}\abs{\sum_{\substack{b\in B_a,\\ c\in C}}\chi(b+c+xy)}}^2\ll\skob{\abs{A_0}\abs I\abs B\abs C}^2r2^{\frac {3d}{8r}} L^{\frac{2}{3r}}p^{\frac{3d\alpha}{4r}-\frac{31\delta}{24r}}\log^{\frac1{4r}} p
    \,.
\end{equation*}
Because $\alpha = \frac{7\delta}{18 d}$ and hence  $r\ge \frac1\alpha = \frac{18d}{7\delta}$, we obtain further 
\begin{equation}
\skob{\sum_{x\in A_0, y\in I}\abs{\sum_{\substack{b\in B_a,\\ c\in C}}\chi(b+c+xy)}}^2
\ll\skob{\abs{A_0}\abs I\abs B\abs C}^2\frac d\delta L^{\frac{7\delta}{27}}p^{-\frac{\delta}{r}}\log^{\frac1{4r}} p \,.
\label{3.24}
\end{equation}
Bound~(\ref{3.24}) takes place for any $a$ and thus inequalities  (\ref{A-A_0I}), (\ref{rand})
imply 
\begin{multline}
\abs{\sum_{\substack{a\in A,b\in B,\\ c\in C}}\chi(a+b+c)}\ll \sqrt{\frac d\delta} L^{\frac{7\delta}{54}}p^{-\frac{\delta}{2r}}\abs{A-A_0I}\abs{B}\abs{C}\log^{\frac1{8r}} p\ll\\ \ll \sqrt{\frac d\delta} L^{\frac{7\delta}{54}}2^de^{C(K)}p^{-\frac{7\delta^2}{72d}}\abs A\abs{B}\abs{C}\log^{\frac1{8r}} p \,.
\label{final_inequality}
\end{multline}
The theorem follows from (\ref{final_inequality})  if one takes $\tau=\frac{\delta^2}{100(C(K)+1)}$, for example. 
\end{proof}

\begin{note}
From inequality (\ref{final_inequality}) it is easy to find the quantity $p(\delta, K,L)$ in a concrete form.
Indeed, it is enough to choose $p$ such that
$\log p \gg \frac{C^2(K)}{\delta^2}$ and $\log p \gg \frac{C(K) \log L }{\delta}$.
It shows that we have subexponential dependence of the constants $K, L$ on $p$ in our theorem.
\label{p(d,K,L)}
\end{note}

\bigskip

\begin{proof}[The proof of the main theorem]
Let  $M>0$ be a real parameter which we will choose later.
Put $\eps=M \sqrt{\frac{\log 2K}{\delta\log p}}$.
Using Lemma \ref{Croot-Sisask} of Croot and Sissak with $q=2$ and $S=A$, $f=B$, we find $a\in A$ 
and a set $T\subset A-a$ such that $\abs{T}\ge \abs{A} \cdot \exp( -\eps^{-2} \log 2K)$ and for any $t\in T$ one has
\begin{equation*}
\|(A\ast B)(x+t)-(A\ast B)(x)\|_{2}\le\eps\abs{A}\abs{B}^{\frac12}.
\end{equation*}
Clearly, the cardinality of the support of the function $(A\ast B)(x+t)-(A\ast B)(x)$ 
does not exceed $2\abs{A+B}$ and hence by the H\"{o}lder inequality the following holds
\begin{multline}\label{main1}
\|(A\ast B)(x+t)-(A\ast B)(x)\|_{1}\le\\
\le\|(A\ast B)(x+t)-(A\ast B)(x)\|_{2}\skob{2\abs{A+B}}^{\frac12}\le\eps\skob{2L}^{\frac12}\abs{A}\abs{B}.
\end{multline}
The constant $M$ in the definition of $\eps$ can be chosen in such a way that $\abs T>p^{\frac{12}{31} +\frac{\delta}2}$.
Besides by the Pl\"{u}nnecke--Ruzsa triangle inequality, we get 
$$\abs{B-T}\le\frac{\abs{B+A}\abs{A+A}}{\abs A}\ll KL\abs B \,.$$
Thus the sets $A$, $B$ and $-T$ satisfy all conditions of Theorem \ref{ternarnic} with $A=A$, $B=B$ and $C=-T$.
Taking $p> p(\delta,K,L)$, we obtain
\begin{equation*}
\abs{\sum_{\substack{a\in A,\, b\in B,\, t\in T}}\chi(a+b-t)}=\abs{\sum_{t\in T,\, x\in \mathbb{F}_p}(A\ast B)(x+t)\chi(x)}<p^{-\tau}\abs A\abs B\abs T \,,
\end{equation*}
where $\tau=\tau(\delta, K)=\delta^2(\log 2K)^{-3+o(1)}$.
Whence for all sufficiently large $p$, namely, for 
\begin{equation}\label{f:cond_p,K_t4}
    \log p / \log \log p \gg\delta^{-2}(\log 2K)^{3+o(1)} \,,
\end{equation}
the inequality
\begin{equation*}
\tau \log p
\gg -\log (\eps L^{1/2}) \,,
\end{equation*}
takes place and thus 
\begin{equation}\label{main2}
\abs{\sum_{t\in T,\, x\in \mathbb{F}_p}(A\ast B)(x+t)\chi(x)}\le\eps L^{1/2} \abs A\abs B\abs T.
\end{equation}
Now, using bounds (\ref{main1}),~(\ref{main2}) and the triangle inequality,  we get 
\begin{multline}
\abs T\abs{\sum_{a\in A,\, b\in B}\chi(a+b)}=\abs{\sum_{t\in T,\, x\in \mathbb{F}_p}(A\ast B)(x)\chi(x)}=\\
=\abs{\sum_{t\in T,\, x\in \mathbb{F}_p}(A\ast B)(x+t)\chi(x)+\sum_{t\in T,\, x\in \mathbb{F}_p}((A\ast B)(x)-(A\ast B)(x+t))\chi(x)}\le\\
\le\abs{\sum_{t\in T,\, x\in \mathbb{F}_p}(A\ast B)(x+t)\chi(x)}\,+\sum_{t\in T}\|(A\ast B)(x+t)-(A\ast B)(x)\|_1\le\\
\le
    4\eps L^{\frac12} \abs A\abs B\abs T \,,
\end{multline}
hence
\begin{equation}
\abs{\sum_{a\in A,\, b\in B}\chi(a+b)}\le 4\eps L^{\frac12} \abs A\abs B\ll\sqrt{\frac{L\log 2K}{\delta\log p}}\abs A\abs B \,.
\end{equation}
This completes the proof of the theorem.
\end{proof}

\bigskip

In the beginning of writing the text we planed to use Burgess inclusion   (\ref{f:Bergess}) in the form
$$
    T+\{1,2,\dots,k\}\cdot T \subseteq (k+1)T \,,
$$
where the set of almost periods $T$ is given by the Croot--Sisask lemma.
Nevertheless it turns out that the arguments above are more effective.

\bigskip

We finish the paper showing how our Theorem \ref{ternarnic} implies Theorem \ref{t:Hanson}.

\bigskip

\begin{proof}[The scheme of the proof of Theorem \ref{t:Hanson}]
    We almost repeat the arguments from \cite{Hanson}.
    Assuming 
$$
    \abs{\sum_{a\in A,\, b\in B,\, c\in C}\chi(a+b+c)} \ge \eps |A||B||C|
$$
    one can easily derive from it that 
$$
    \E^{+} (B,C) := |\{ b+c = b'+c' ~:~ b,b' \in B,\, c,c' \in C \}| \gg (\eps \zeta)^2 (|B||C|)^{3/2}
$$
    and 
$$
    \E^{+} (A) \gg (\eps \zeta)^4 |A|^3 \,.
$$
    After that we use the Balog--Szemer\'{e}di--Gowers Theorem, see e.g. \cite{TV} and find subsets $A'\subseteq A$, $B'\subseteq B$, $C'\subseteq C$ such that 
    $|A'+A'| \ll (\eps \zeta)^{-M} |A'|$, $|B'+C'| \ll (\eps \zeta)^{-M} (|B'||C'|)^{1/2}$ and 
    $|A'| \gg (\eps \zeta)^{M} |A|$, $|B'| \gg (\eps \zeta)^{M} |B|$,  $|C'| \gg (\eps \zeta)^{M} |C|$.
    Here  $M>0$ is an absolute constant.
    Applying Theorem \ref{ternarnic} to the obtained sets and using simple average arguments (see \cite{Hanson}), we arrive to a contradiction.

    It is easy to count (see, e.g. condition (\ref{f:cond_p,K_t4}) from the proof of Theorem \ref{ternarnic} or Remark \ref{p(d,K,L)}) that a nontrivial estimate in formula (\ref{f:Hanson}) requires  the restriction of  the form 
$\zeta \gg \exp(-(\log p)^{\alpha})$, where $\alpha >0$ is an absolute constant.
\end{proof}

\bigskip

\noindent{A.S.~Volostnov\\
Moscow Institute of Physics and Technology (State University),\\
9 Institutskiy per., Dolgoprudny, Moscow Region, 141700, Russian Federation}
{\tt gyololo@rambler.ru}

\bigskip

\noindent{I.D.~Shkredov\\
Steklov Mathematical Institute of Russian Academy of Sciences,\\
ul. Gubkina, 8, Moscow, Russia, 119991}
\\
{\tt ilya.shkredov@gmail.com}

\end{document}